\documentclass{article}
\usepackage{amssymb}

\newtheorem{theorem}{Theorem}

\newenvironment{proof}[1][Proof]{\noindent\textbf{#1.} }{\ \rule{0.5em}{0.5em}}
\input{tcilatex}

\begin{document}

\title{An information-theoretic analog of a result of Perelman}
\author{Brockway McMillan \\
Sedgwick Maine, USA (e-mail: bmcmlln@hughes.net)}
\maketitle

\begin{abstract}
Each compact manifold $\mathbf{M}$\ of finite dimension $k\mathbf{\ }$is
differentiable and supports an intrinsic probability measure. There then
exists a measurable transformation of $\mathbf{M\ }$to the $k$-dimensional
"surface" of the $(k+1)-$dim- ensional ball.
\end{abstract}

\section{Manifolds}

\subsection{Topologies, coordinates, and measures}

Let $k$ be a positive integer. By definition\emph{,} a given compact
topological space $\mathbf{M\ }$is a \emph{manifold of} $k\ $\emph{dimensions%
}.if every point $p\in \mathbf{M}$ has a neighborhood that is topologically
equivalent to a Euclidean open sphere, of $k$ dimensions, centered at$\ p.\ $%
Call such a neighborhood a \emph{cell centered at} \emph{\ }$p,.$or, less
formally, a \emph{cell}.\ 

To a given cell $C\ $centered at $p$\ then corresponds a topological
transformation $T,\ $specific to $C,\ $that transforms a Euclidean $k$%
-sphere, and therefore also transfers the coordinate axes defined therein,
to a topological space $C\subset \mathbf{M}\mathbf{.}$

On the Euclidean space $T^{-1}C\ $there exists a $\sigma $-algebra $\mathcal{%
F}_{E}\ $of measurable sets-- the smallest $\sigma $-algebra that contains
every open set.\ Because $T$\ \ takes open sets to open sets\ $T\mathcal{F}%
_{E}\ $is also a $\sigma $-algebra on$\ C.$ Then $T\ $\/is what is often
called a \emph{measurable transformation}. It transfers the Lebesgue measure 
$\lambda (\cdot )\ $defined on $T^{-1}C\ $to a measure $\mu (\mathbf{\cdot }%
)=\lambda (T^{-1}\mathbf{\cdot })\ $on $C.\ $

As the image of a Euclidean sphere, a cell centered at $p\mathbf{\ }$can be
given a quasi-Cartesian coordinate frame, specific to that cell, with origin
at $p$\ and identified by its\ (curviliear) coordinates, say\ $%
x_{1},x_{2},\cdots ,x_{k}.\smallskip $

$\mathbf{M\ }$is compact and is covered by finitely many cells.$\mathbf{\ }$%
Such a covering is a \textit{proper covering.\smallskip }

\ Let $C_{1},C_{2},$ be distinct cells of a given proper covering, endowed
respectively with coordinate frames $\{x_{r}\},\ \{y_{s}\}$\textbf{.} Let $%
p\ $be a point in $C_{1}\cap C_{2}.\ $Because $\mathbf{M}$\ is a manifold,
the coordinates $\{x_{r}\}\ $of $p$ are already constrained to be continuous
functions of the coordinates $\{y_{s}\}\ $of $p,.$and \textit{vice versa.\ }%
The rest of the paper hinges on the existence of a much stronger
constraint.\smallskip

\begin{theorem}
\label{thm1} The compact manifold $\mathbf{M\ }$above is \emph{%
differentiable }in the sense that the $\{x_{r}\}\ $are differentiable
functions of the $\{y_{s}\},\ $and \textit{vice versa}\textbf{.}\ \smallskip
\end{theorem}

\noindent A proof appears in the next subsection.\smallskip

The compactness of $\mathbf{M\ }$implies that any covering of $\mathbf{M\ }$%
by cells contains a subcovering by finitely many cells. Such a subcovering,
by definition, is a \emph{proper covering.\smallskip }

A point in $\mathbf{M\ }$with coordinates $\{x_{r}\}\mathbf{\ }$is often
called a \emph{vector,} and $\mathbf{M}\ $then treated as a linear space.
That vector space is a \emph{Riemannian manifold.} On it exists its tensor
calculus\emph{,\ }a powerful tool, somewhat complex. The arguments here do
not require a linear structure on $\mathbf{M}$\textbf{\ }and can then
fortunately ignore the tensor calculus. \smallskip

\subsection{The intrinsic measure of volume}

Let $C\in \mathbf{M}\ $be a cell centered at a point $p.$\ $C$ is the image
under a topological transformation of an open Euclidean sphere. $C$\ is what
is often called a \emph{measurable space},-- it supports a $\sigma $-algebra$%
\ \mathcal{F}_{C}$ of measurable sets,-- the smallest $\sigma $-algebra that
contains every open set. On that $\sigma $-algebra there exists a
well-defined measure-- a "local" volume defined on a cell of a $k$%
-dimensional differentiable manifold\textbf{. }In terms of the local
coordinates $\{x_{r}\}$ that volume is defined by its local \emph{density} $%
dx_{1}dx_{2}\cdots dx_{k}.\ $

This volume element is a function of position in the cell on which it is
defined. The measure, say\ $\mu (\sigma ),\ $of a set $\sigma \in $ $%
\mathcal{F}_{C}\ $is the integral over $\sigma \ $of that density function.\
Furthermore, at that same point, but in an overlapping cell, there is a
volume element, say $dy_{1}dy_{2}\cdots dy_{k}.\ $These two differential
volume elements are \emph{equal},--\emph{\ }they are in fact the volume
element of the (Euclidean) tangent space to\ $\mathbf{M\ }$at that point.
They are simply transferred from the Euclidean pre-image of the cell at
issue. They therefore do not vanish at any point of $\mathbf{M}$. It follows
that there exists on the whole of $\mathbf{M\ }$a $\sigma $-algebra of
measurable subsets with a measure $\mu (\cdot )$ defined thereon. That $%
\sigma $-algebra is the smallest $\sigma $-algebra that contains every open
set. (Being an algebra, it then also contains every closed set.) Since every
proper covering is made up of finitely many cells, each of finite volume, $%
\mu (\mathbf{M)<\infty .}$ One may then choose to normalize it to the value 
\textbf{\ }$\mu (\mathbf{M)=1\ }$and call $\mu $ a probability.

Whether or not normalized, the $\mu \ $defined by this construction is an 
\emph{intrinsic measure }on $\mathbf{M,\ }$or \emph{intrinsic probability}
if normalized. \smallskip

\subsubsection{Summary: \textbf{R-}measures}

By definition an \textbf{R}-measure on$\ \mathbf{M}\ $is a measure $\mu \ $%
that enjoys the following properties\textbf{:}

\begin{itemize}
\item $0<\mu (\mathbf{M)<\infty .}$

\item $\mu \ $is \emph{smooth\ }in that each point of $\mathbf{M\ }$is a set
of measure zero$.$

\item If $A\ $is a non-empty open subset of\ $\mathbf{M\ }$then $\mu (A)>0.$

\item It follows that the discrete subsets of $\mathbf{M\ }$constitute the
totality of null sets of $\mu \mathbf{.}$

\item It then further follows that any two \textbf{R}-measures on\ $\mathbf{%
M\ }$are \emph{compatible\ }in that each is absolutely continuous with
respect to the other.\medskip
\end{itemize}

\begin{theorem}
The intrinsic measure on a given differentiable $k$-manifold\ $\mathbf{M\ }$%
is a topological invariant and is an \textbf{R}-measure\textbf{. }

\begin{proof}
First address invariance. Let $T$\ \/be a continuous transformation from $%
\mathbf{M\ }$onto $\mathbf{M}^{\prime }=T\mathbf{M.\ }T\ $maps open sets\ to
open sets\ and is therefore also a measurable transformation. If $\mu \ $is
a measure on $\mathbf{M\ }$then $\mu \mathbf{\prime }(A)=\mu (T^{-1}A)\ $is
a measure on $\mathbf{M}^{\mathbf{\prime }}.\ $It is then easy \ to verify
that all five bullets above apply to $\mathbf{M}^{\prime }$ \medskip
\end{proof}
\end{theorem}

\section{Entropy}

Consider a \emph{finite} partition of the space $\mathbf{M}\ $into
pairwise-disjoint measurable sets. Say $\pi \ $is one such partition:%
\begin{equation}
\pi :\ \mathbf{M=\oplus }_{k=1}^{card(\pi )}C_{k}.  \label{pi}
\end{equation}%
With this partition Claude Shannon [CS] associates the quantity%
\begin{equation}
H(\pi )=\frac{1}{card(\pi )}\tsum\nolimits_{C\in \pi }\mu (C)\log \frac{1}{%
\mu (C)}.  \label{DefEnt}
\end{equation}%
He then considers a refining sequence of partition $\pi _{1}\succ \pi
_{2}\succ \cdots \ $and shows that the quantities $H(\pi _{n})\ $converge in
probability to a limit $H(\mathbf{M})\mathbf{\ }$that is independent of the
chosen sequence. He calls $H(\mathbf{M})\mathbf{\ }$ the \emph{entropy }of
the measure $\mu .\ $It was shown in [Mc1] that these $H(\pi _{n})\ $also
converge in $\pounds ^{1}\ $mean. Stronger convergence theorems were soon
proved by others.\ By about 1980, with the work of Kolmogorov and
colleagues, and finally of \ D.S. Ornstein, Shannon's full theory of
communication became a closed book .That theory includes much more than is
reported on here; specifically it also presents what is usually known as
Shannon's Coding Theorem.\medskip

\subsubsection{\textbf{Nomenclature}}

During development of his theory, Shannon was reluctant to use the term
"entropy". At the urging of colleagues, he finally relented. As he feared,
the term "entropy" spawned much nonsense. At a meeting of the American
Physical Society in 1950 one member of the large audience announced that
"Claude Shannon has proved that a heat engine can do mathematical logic." \
(I was there. I heard it. I recognized the speaker but fortunately no longer
recall his name.)\medskip

\subsection{Simple properties\label{SP}}

(\textbf{1}) Let$\ \mathbf{M}_{1},\mathbf{M}_{2},\ $be distinct instances of
the generic compact differentiable manifold\ $\mathbf{M},\ $not necessarily
of the same dimension. It follows that $\mathbf{M}_{1}\otimes \mathbf{M}%
_{2}\ $is also an instance. By calculation, directly from (\ref{DefEnt}),
the entropy of their cartesian product is%
\begin{equation}
H(\mathbf{M}_{1}\otimes \mathbf{M}_{2})=H(\mathbf{M}_{1})+H(\mathbf{M}%
_{2}).\medskip  \label{Cart}
\end{equation}

(\textbf{2}) Return to the generic compact and differentiable manifold $%
\mathbf{M.}$ Let\ $\mu \ $be an \textbf{R}-measure on $\mathbf{M\ }$and let $%
\pi \ $be the finite partition (\ref{pi}). Define%
\begin{equation}
\Lambda (\mu ,\pi )=\tsum\nolimits_{C\in \pi }\mu (C)\log \frac{1}{\mu (C)}%
-\mu (\mathbf{M})\log \frac{1}{\mu (\mathbf{M})}.  \label{deflambda}
\end{equation}%
Let $\alpha >0.\ $One calculates that$\ \Lambda (\alpha \mu ,\pi )=\alpha
\Lambda (\mu ,\pi ).\ $Consequently, if $\alpha ,\beta ,\ $are non-negative
numbers and $\mu _{1},\ \mu _{2},\ $are \textbf{R}-measures on $\mathbf{M\ }$%
then%
\begin{equation}
\Lambda (\alpha \mu _{1}+\beta \mu _{2},\pi )=\alpha \Lambda (\mu _{1},\pi
)+\beta \Lambda (\mu _{2},\pi ),  \label{conv}
\end{equation}%
the property of strict convexity. (See \P \ref{XX} later.)\medskip

\ Let $\pi _{1}\succ \pi _{2}\succ \cdots \ $be a refining sequence of
finite partition of $\mathbf{M\ }$\ into measurable sets.\ By Shannon's
theorem, if $\mu \ $is a probability measure on $\mathbf{M\ }$then 
\[
\lim_{n\rightarrow \infty }(card(\pi _{n}))^{-1}\Lambda (\mu ,\pi _{n})\
=H(\mu ). 
\]%
It follows that

\begin{theorem}
\label{fnl}(\textbf{1}) $H(\mu ),\ $as a function of probability measures $%
\mu \ $defined on $\mathbf{M},\ $is strictly convex.

\noindent (\textbf{2}) There exists a unique probability measure $\mu
_{\infty }\ $on$\ \mathbf{M\ }$that maximizes $H(\cdot ).\ $It is
characterized by the property$\ $that for all measurable subsets $A,B,\ $in $%
\mathbf{M}$ 
\begin{equation}
\mu _{\infty }(A\cap B)=\mu _{\infty }(A)\cdot \mu _{\infty }(B).
\label{ind}
\end{equation}%
\noindent (\textbf{3})$\ $Every probability measure $\mu \ $on $\mathbf{M\ }$%
is compatible with $\mu _{\infty },$ and in particular is absolutely
continuous with respect to $\mu _{\infty }.\mathbf{\smallskip }$
\end{theorem}

\begin{proof}
(\textbf{1}) simply repeats (\ref{conv}).$\ $The proof of (\textbf{2})%
\textbf{\ }involves an excursion into the theory of point processes on a
space such as $\mathbf{M},$\ undertaken in the next section.\medskip
\end{proof}

\section{The Poisson Process on $\mathbf{M}$}

A\ \emph{point process}\ on $\mathbf{M}\ $is a random process\ of which the
generic random variable$\ \gamma \ $\linebreak is a discrete subset of $%
\mathbf{M;\ }$in the present case, that discreteness implies that a random
set is also a finite subset of$\ \mathbf{M.\ }$It is a theorem of Thomas
Kurtz$\ $[TK]\ that any point process is characterized by its \emph{%
avoidance function}%
\[
\mathbf{E(}A)=\Pr ob\{\gamma ~|\ A\cap \gamma =\emptyset \}=\Pr
ob\{card(A\cap \gamma )=0\}. 
\]

\emph{\ }$.\ $\noindent All that one needs to know here is that if $\mu \ $%
is a measure on $\mathbf{M\ }$then the function $\mathbf{E(}A)=e^{-\mu (A)}\ 
$is a valid avoidance function. It is the avoidance function of a point
process that is the counterpart on $\mathbf{M}\ $of the Poisson process on
the real line. This latter is the \emph{discrete homogeneous chaos }of
Norbert Wiener [NW].\ On any space\ that is locally compact and metrizable,
the Poisson process is that unique process for which, for each two
measurable sets $A,B,$ the random sets $A\cap \gamma ,\ B\cap \gamma ,$\ are
statistically independent whenever\ $A\cap B=\emptyset .\ $Conclusions (%
\textbf{2}) and (\textbf{3}) above are explicit in each of Lemma 43 and
Theorem 48 of [Mc2]~$\blacksquare $\smallskip

\section{Perelman's result, discussion \label{XX}}

Perelman shows that the Ricci flow carries a given compact $k$-manifold $%
\mathbf{M}\ $to a terminal $k$-manifold that has a maximum entropy. The
arguments above demonstrate a different way to make a similar association.
In each case the entropy involved is that of Shannon and of statistical
mechanics. The maxentropic manifold in each is a featureless manifold that
is topologically the $k$-dimensional "surface" of a $(k+1)-$dimensional ball%
\textbf{.} \medskip

Subsection \ref{SP} mentions convexity, a matter of no consequence to
Theorem \ref{fnl} above. There is a literally monstrous \textit{%
gedankenexperimente, }attributed to Einstein, showing that some simple
thermodynamic properties, combined with the property of convexity, make the
entropy $H(\cdot )\ $a unique functional.\ In [JvN],\ von Neumann describes
Einstein's argument,-- with some evidence of distaste,-- without, to my
reading, fully closing the issue of uniqueness. Perelman's Ricci flow
clinches that latter issue. Theorem \ref{fnl} above does also, but,\ as an
existence theorem,\ lacks the inevitable force of Perelman's constructive
Ricci flow\textbf{.}\smallskip

\subsection{Credits\ }

I thank Prof. Aubert Daigneault for corrections, suggestions, and patience%
\textbf{.} Thanks also are due Aaron F. McMillan for a review of modern
differential geometry.\smallskip

\subsection{Bibliography\protect\smallskip}

\noindent \lbrack CES]\ Shannon, Claude E. \textit{A mathematical theory of
communication}, Bell System Technical Journal, \textbf{v}.27, pp 379-423, pp
623-656, July-October 1948.

\noindent \lbrack GP] Grisha Perelman \textit{The entropy formula for the
Ricci flow and its geometric applications, }\textbf{Arxiv (}Math/\textbf{DG)}

\noindent \lbrack JvN] John von Neumann, \textit{Quantenmechanik,\ }Dover
NY, 1945. (reprinted from Springer, pre-1940.)

\noindent \lbrack Mc1] Brockway McMillan,\textit{\ The basic theorems of
information theory,} Annals of Mathematical Statistics, V 24 pp 196-219,
June 1953.

\noindent \lbrack Mc2] $\mathbf{\cdots }$, \textit{A taxonomy for random
sets, }International Journal of Pure and Applied Mathematics, \textbf{v 34},
No.3, \textit{pp }347-396.

\textbf{Addendum} (for the author's convenience during the submission
process.)

\ \ \ \ \ \ \ \ \ \ \ \ \ \ \ \ \ \ \ \ \ \ \ \ \ \ \ \ \ Grisha Perelman 
\TEXTsymbol{<}perelman@math.sunysb.edu\TEXTsymbol{>}

\end{document}